\documentclass[12pt,reqno]{amsart}
\usepackage{amscd,amssymb,amsmath,amsthm}
\usepackage{cite}
\newtheorem{thm}[subsection]{Theorem}
\newtheorem{lemma}[subsection]{Lemma}

\newtheorem{cor}[subsection]{Corollary}

\newtheorem{rk}[subsection]{Remark}
\newtheorem{defn}[subsection]{Definition}
\newtheorem{ex}{Example}

\numberwithin{equation}{section} 

\newcommand{\bea}{\begin{eqnarray}}
\newcommand{\eea}{\end{eqnarray}}

\DeclareMathOperator{\Der}{\rm Der}
\begin{document}

\title[Nilpotent Evolution Algebras]
{Derivations and automorphisms of nilpotent evolution algebras
with maximal nilindex}

\author{Farrukh Mukhamedov}
\address{Farrukh Mukhamedov\\
 Department of Mathematical Sciences\\
College of Science, The United Arab Emirates University\\
P.O. Box, 15551, Al Ain\\
Abu Dhabi, UAE} \email{{\tt far75m@gmail.com} {\tt
farrukh.m@uaeu.ac.ae}}

\author{Otabek Khakimov}
\address{Otabek Khakimov\\
Department of Mathematical Sciences\\
 College of Science, United Arab Emirates University\\
P.O. Box, 15551, Al-Ain\\
Abu Dhabi, UAE} \email{{\tt hakimovo@mail.ru}
{\tt
otabek.k@uaeu.ac.ae}}

\author{Bakhrom Omirov}
\address{Bakhrom Omirov\\
National University of Uzbekistan\\
4, University street, 100174, Tashkent,
Uzbekistan} \email{{\tt omirovb@mail.ru}}

\author{Izzat Qaralleh}
\address{Izzat Qaralleh\\
Department of Mathematics\\
Faculty of Science, Tafila Technical
University\\
Tafila, Jordan}
\email{{\tt izzat\_math@yahoo.com}}

\begin{abstract}

In this paper is devoted to nilpotent finite-dimensional evolution algebras $\bf E$
with $\dim{\bf E^2}=\dim{\bf E}-1$. We described Lie algebras associated with evolution algebras
whose nilindex is maximal. Moreover, in terms of this Lie algebra we fully construct nilpotent evolution algebra with maximal
index of nilpotency. Furthermore, this result allowed us fully characterize all local and 2-local derivations of the considered
evolution algebras. All automorphisms and local automorphisms of the nilpotent evolution algebras
with maximal nilindex are found.
\end{abstract}

\maketitle

\section{Introduction}

Recently in \cite{tv} a new type of evolution algebra is introduced. This algebra also describes
some evolution laws of the genetics. The study of evolution algebras constitutes a new subject both in algebra and
the theory of dynamical systems. There are many related open problems to promote further research in this subject (for more details we refer to \cite{t}).

We notice that evolution algebras are not defined by identities, and therefore they do not form a variety of non-associative
algebras, like Lie, Jordan or alternative algebras. Hence, the investigation of such kind of algebras needs
a different approach (see \cite{Some_properties,derevol,rozomir}).

In \cite{rozomir} the equivalence between nil, right nilpotent evolution algebras and evolution algebras, which are defined by an upper triangular matrix of structural constants, have been established. A classification of low dimensional evolution algebras have
been carried out in\cite{3dim,heg1,heg2,Elduque}. However, a full classification of nilpotent evolution algebras
is far from its solution. Therefore, in the present paper we are going to investigate
certain properties of nilpotent evolution algebras with maximal nilindex.

It is known that in the theory of non-associative algebras, particularly, in genetic algebras,
the Lie algebra of derivations of a given algebra is one of the important tools for studying its structure.
There has been much work on the subject of derivations of genetic algebras (\cite{Costa1}, \cite{Costa2}, \cite{Gonshor}, \cite{Holgate}).

In fact, in \cite{derevol} the authors investigate several properties of derivations of $n$-dimensional
complex evolution algebras, depending on the rank of the appropriate matrices. In the present paper
we explicitly describe the space of derivations of evolution algebras with  maximal nilindex which allows us
to study further properties of the evolution algebras. Moreover, we describe all local and 2-local derivations of
the considered algebra. We stress that the notions of local automorphism and local derivation were introduced and investigated independently by
Kadison \cite{Kadison} and Larson and Sourour \cite{Larson}. Later, in 1997, P. $\check{S}$emrl \cite{Semrl} introduced the concepts of $2$-local automorphisms and $2$-local derivations. The above papers gave rise to series of works devoted to description of mappings which are close to automorphisms and derivations of $C^*$-algebras and operator algebras. For details and the survey we refer to the paper \cite{Ayupov}.

The paper is organized as follows. In Section 2 we provide preliminary information about evolution algebras.
It is well known that derivations of non-associative algebras form
Lie algebra, so, in Section 3 we describe the Lie algebra associated with evolution algebras
whose nilindex is maximal. Moreover, in terms of this Lie algebra we fully construct nilpotent evolution algebra with maximal
index of nilpotency. Furthermore, in Section 4 by means of result Section 3 we describe local and 2-local
derivations of the considered
evolution algebras. In Section 5 we find all automorphisms and local automorphisms of the nilpotent evolution algebras
with maximal nilindex.

\section{Evolution algebras}

Recall the definition of evolution algebras. Let $\bf{E}$ be a vector space over a field $\mathbb K$.
In what follows, we always assume that $\mathbb K$ has characteristic zero. The vector space $\bf{E}$
is called {\it evolution algebra} w.r.t. {\it natural basis} $\{{\bf e}_1, {\bf e}_2, . . . \}$ if
a multiplication rule $\cdot$· on $\bf{E}$ satisfies
$$
{\bf e}_i\cdot {\bf e}_j={\bf 0},\ i\neq j,
$$
$$
{\bf e}_i\cdot{\bf e}_i=\sum_{k}a_{ik}{\bf e}_k,\ i\geq1.
$$

From the above definition it follows that evolution algebras are commutative (therefore, flexible).

We denote by $A=(a_{ij})^n_{i,j=1}$ the matrix of the structural constants
 of the finite-dimensional evolution algebra $\bf{E}$.
Obviously, $rank A =\dim(\bf{E}\cdot\bf{E})$. Hence, for finite-dimensional evolution algebra the
rank of the matrix does not depend on choice of natural basis.

In what follows for convenience, we write ${\bf u}{\bf v}$ instead ${\bf u}\cdot{\bf v}$
for any ${\bf u},{\bf v}\in\bf{E}$ and we shall write $\bf{E}^2$ instead $\bf{E}\cdot\bf{E}$.

A linear map $\psi: {\bf E}_1\to{\bf E}_2$ is called an {\it homomorphism} of evolution algebras
if $\psi({\bf u}{\bf v})=\psi({\bf u})\psi({\bf v})$ for any ${\bf u},{\bf v}\in{\bf E}_1$. Moreover, if $\psi$ is bijective, then it is called an
{\it isomorphism}. In this case, the last relation is denoted by ${\bf E}_1\cong{\bf E}_2$.

For an evolution algebra $\bf E$ we introduce the following sequence, $k\geq1$
\begin{equation}\label{E^k}
{\bf E}^{k}=\sum_{i=1}^{k-1}{\bf E}^{i}{\bf E}^{k-i}.
\end{equation}
Since ${\bf E}$ is commutative algebra we obtain
$$
{\bf E}^{k}=\sum_{i=1}^{\lfloor k/2\rfloor}{\bf E}^{i}{\bf E}^{k-i},
$$
where $\lfloor x\rfloor$ denotes the integer part of $x$.

\begin{defn}
An evolution algebra ${\bf E}$ is called nilpotent if there exists some $n\in\mathbb N$
such that ${\bf E}^m=\bf 0$. The smallest $m$ such that ${\bf E}^m=\bf 0$ is called the index of nilpotency.
\end{defn}

\begin{thm}\label{thm_ro}\cite{rozomir}
An $n$-dimensional evolution algebra $\bf E$ is nilpotent iff it admits a natural basis such that the matrix of the
structural constants corresponding to $\bf E$ in this basis is represented in the form
$$
\tilde{A}=\left(
\begin{array}{lllll}
0 & \tilde{a}_{12} & \tilde{a}_{13} & \vdots  & \tilde{a}_{1n}\\
0 & 0 & \tilde{a}_{23} & \vdots  & \tilde{a}_{2n}\\
\vdots & \vdots & \vdots & \ddots & \vdots\\
0 & 0 & 0 & \vdots & \tilde{a}_{n-1,n}\\
0 & 0 & 0 & \vdots & 0
\end{array}\right)
$$
\end{thm}
Due to Theorem \ref{thm_ro} any nilpotent evolution algebra $\bf{E}$ with $\dim({\bf{E}}^2)=n-1$ has the following form:
\begin{equation}\label{evolalg}
{\bf e}_i^2=\left\{
\begin{array}{lll}
\sum\limits_{j=i+1}^na_{ij}{\bf e}_j, & i\leq n-1;\\
{\bf 0}, & i=n.
\end{array}
\right.
\end{equation}
where $a_{ij}\in\mathbb K$ and $a_{i,i+1}\neq0$ for any $i<n$.

\begin{thm}\cite{clor}
Let $\bf E$ be a nilpotent evolution algebra. Then $\bf E$ has maximal index of
nilpotency $2^{n-1} + 1$, if and only if the multiplication table of $\bf E$ is given by
\eqref{evolalg}.
\end{thm}

In what follows, we are going to work with nilpotent evolution algebras with
maximal index of nilpotency. Due to last theorem we only consider
evolution algebras with multiplication table given by \eqref{evolalg}.

\begin{lemma}\label{lem1}
Let $\bf{E}$ and ${\bf{E}}'$ be evolution algebras with basis $\{{\bf e}_i\}_{i=1}^n$ and
$\{{\bf f}_i\}_{i=1}^n$ respectively, defined by
$$
{\bf e}_i^2=\left\{
\begin{array}{lll}
a_{i,i+1}{\bf e}_{i+1}+a_{in}{\bf e}_n, & i<n-1;\\
a_{n-1,n}{\bf e_{n}}, & i=n-1;\\
{\bf 0}, & i=n.
\end{array}
\right.
\ \ \ \ {\bf f}_i^2=\left\{
\begin{array}{lll}
{\bf f}_{i+1}, & i<n;\\
{\bf 0}, & i=n.
\end{array}
\right.
$$
If $a_{i,i+1}\neq0$ for every $i<n$,
then ${\bf E}\cong{\bf E}'$.
\end{lemma}

\begin{proof} Let $a_{i,i+1}\neq0$ for every $i<n$. If $n=2$
after changing the basis ${\bf e}_1,{\bf e}_2$ to ${\bf f}_1={\bf e}_1$ and ${\bf f}_2={\bf e}_1^2$
we immediately get ${\bf{E}}'$.

So. let us suppose $n\geq3$.
Then the linear mapping $\varphi: {\bf E}\to{\bf E}'$ defined by
\begin{equation}\label{iso}
\varphi:\ \left\{
\begin{array}{lllll}
{\bf f}_1={\bf e}_1\\
{\bf f}_2={\bf e}_1^2\\
{\bf f}_{i+1}=\prod\limits_{k=1}^{i-1}a_{k,k+1}^{2^{i-k}}{\bf e}_i^2, & 2\leq i<n
\end{array}
\right.
\end{equation}
is an isomorphism from $\bf E$ to ${\bf E}'$.
\end{proof}

\section{Derivations}

In this section, we consider derivations of
nilpotent evolution algebras
with maximal index of nilpotency.

Recall that derivation of an evolution algebra $\bf{E}$ is a linear mapping
$d : \bf{E}\to\bf{E}$ such that
$d({\bf u}{\bf v}) = d({\bf u}){\bf v} + {\bf u}d({\bf v})$
for all ${\bf u}, {\bf v}\in\bf{E}$.

We note that for any algebra, the space $\Der(\bf{E})$ of all derivations is a Lie algebra w.r.t.
the commutator multiplication:
$$
[d_1,d_2]=d_1d_2-d_2d_1,\ \ \ \forall d_1,d_2\in\Der({\bf E}).
$$

\begin{lemma}\label{lem2}
Let ${\bf E}_1$, ${\bf E}_2$ be two isomorphic evolution algebras. Then $\Der({\bf E}_1)\cong\Der({\bf E}_2)$.
\end{lemma}

\begin{proof}
Let $\varphi$ be an isomorphism from ${\bf E}_1$ to ${\bf E}_2$. It is easy to check
that a linear mapping $\psi$ defined on $\Der({\bf E}_1)$ by
$$
\psi(d)=\varphi(d)\varphi^{-1}
$$
is an isomorphism of Lie algebras $\Der({\bf E}_1)$ and $\Der({\bf E}_2)$.
\end{proof}

For a given structural matrix $A=(a_{ij})_{i,j\geq1}^n$ of nilpotent evolution algebra
$\bf E$ with $dim({\bf E}^2)=n-1$ we denote
\begin{equation}\label{A_a_ijneq0}
I_A=\{(i,j):i+1<j<n,\ a_{ij}\neq0\}.
\end{equation}

\begin{thm}\label{thm_der}
Let $\bf{E}$ be an evolution algebra with structural matrix $A=(a_{ij})_{i,j\geq1}^n$ in a
natural basis
$\{{\bf e}_i\}_{i=1}^n$. If $\bf E$ is a nilpotent with
$rank A=n-1$, then the following statements hold
\begin{enumerate}
\item[$(i)$] if $I_A\neq\emptyset$ then
$$
\Der({\bf{E}})=\left\{
\left(
\begin{array}{lllll}
0 & 0 & \ldots & 0 & \beta\\
0 & 0 & \ldots & 0 & 0\\
\vdots & \vdots & \ddots & \vdots & \vdots\\
0 & 0 & \ldots & 0 & 0\\
0 & 0 & \ldots & 0 & 0
\end{array}
\right):\ \beta\in\mathbb K
\right\}
$$
\\

\item[$(ii)$]
if $I_A=\emptyset$ then
$$
\Der({\bf{E}})=\left\{
\left(
\begin{array}{lllll}
\alpha & 0 & \ldots & 0 & \beta\\
0 & 2\alpha & \ldots & 0 & (2-2^{n-1})\alpha a_{1n}\\
\vdots & \vdots & \ddots & \vdots & \vdots\\
0 & 0 & \ldots & 2^{n-2}\alpha & (2^{n-2}-2^{n-1})\alpha a_{n-2,n}\\
0 & 0 & \ldots & 0 & 2^{n-1}\alpha
\end{array}
\right):\ \alpha,\beta\in\mathbb K
\right\}
$$
\end{enumerate}
\end{thm}

\begin{proof} The $(i)$ and $(ii)$ are easy to check for $n=2,3$. So, we consider only the case $n>3$.
Let $d$ be a derivation. We represent $d$ in a matrix form in the basis $\{{\bf{e}}_i\}_{i=1}^n$
as follows $d({\bf{e}}_i)=\sum_{j=1}^nd_{ij}{\bf{e}}_j$.
Then, we have $d_{ji}{\bf e}_i^2+d_{ij}{\bf e}_j^2={\bf 0}$
for all $1\leq i<j\leq n$. Since ${\bf e}_i^2$ and ${\bf e}_j^2$ are linearly independent,
then $d_{ij}=d_{ji}=0$ for any $1\leq i<j<n$. If we take $j=n$,
then taking into account that ${\bf e}_n^2=\bf 0$ from $d_{ni}{\bf e}_i^2+d_{in}{\bf e}_n^2={\bf 0}$ one has $d_{ni}=0$ for any $i<n$.

Hence, we have shown the following:
\begin{equation}\label{d_ij=0}
d_{ij}=0,\ \ \ \mbox{if }i\neq j,\ i\leq n,\ j<n
\end{equation}
On the other hand, we have $d({\bf e}_{i}^2)=2d_{ii}{\bf e}_i^2$ for any $i\leq n$.
Then, for $i=n-1$ using \eqref{evolalg} we obtain
$d(a_{n-1,n}{\bf e}_n)=2d_{n-1,n-1}a_{n-1,n}{\bf e}_n$. Due to $a_{n-1,n}\neq0$ one gets
\begin{equation}\label{d_nn}
d_{nn}=2d_{n-1,n-1}
\end{equation}

Furthermore, assume that $i<n-1$.
Then, one finds
\begin{eqnarray}\label{qwe}
d({\bf e}_i^2)&=&d\left(\sum\limits_{j=i+1}^na_{ij}{\bf e}_j\right)=\sum\limits_{j=i+1}^na_{ij}d({\bf e}_j)\nonumber\\
&=&\sum_{j=i+1}^{n-1}a_{ij}d_{jj}{\bf e}_{j}+\sum_{j=i+1}^na_{ij}d_{jn}{\bf e}_n
\end{eqnarray}
On the other hand, from
$$
d({\bf e}_i^2)=2d_{ii}{\bf e}_i^2=2d_{ii}\sum\limits_{j=i+1}^na_{ij}{\bf e}_j.
$$
with \eqref{qwe} one finds

\begin{eqnarray}
\label{t1}&& 2d_{ii}=d_{i+1,i+1},\ \  1\leq i<n-1\\[2mm]
\label{t2}&& a_{ij}d_{jj}=2a_{ij}d_{ii},\ \ i+2\leq j\leq n-1\\[2mm]
\label{t3}&& \sum\limits_{j=i+1}^na_{ij}d_{jn}=2d_{ii}a_{in},\ \ 1\leq i<n-1.
\end{eqnarray}

From \eqref{t1},\eqref{t2} we can easily derive

\begin{eqnarray}
\label{tt1}d_{jj}=2^{j-1}d_{11}, & 2\leq j\leq n-1\\
\label{tt2}a_{ij}d_{11}=0, & i+2\leq j\leq n-1.
\end{eqnarray}

Now we consider \eqref{t3}. We claim:
\begin{equation}\label{d_i+1,n=matind}
d_{i+1,n}=\frac{(2^i-2^{n-1})d_{11}a_{in}}{a_{i,i+1}},\ \ \ 1\leq i\leq n-2.
\end{equation}
Let us prove it by induction.
Let $i=n-2$. Then, from \eqref{t3} one gets
$$
a_{n-2,n-1}d_{n-1,n}=(2d_{n-2,n-2}-d_{nn})a_{n-2,n}.
$$
Noting $a_{n-2,n-1}\neq0$ and plugging \eqref{d_nn} and \eqref{tt1} into the last
one, we obtain
\begin{equation}\label{d_n-1,n=}
d_{n-1,n}=\frac{(2^{n-2}-2^{n-1})d_{11}a_{n-2,n}}{a_{n-2,n-1}}.
\end{equation}
Now assume that \eqref{d_i+1,n=matind} holds for any $2\leq i\leq n-2$. Then from
\eqref{t3} for any $2\leq i\leq n-2$ one finds
\begin{eqnarray}\label{d_in}
a_{i-1,i}d_{i,n}&=&(2d_{i-1,i-1}-d_{nn})a_{i-1,n}-\sum\limits_{j=i+1}^{n-1}a_{i-1,j}d_{j,n}\nonumber\\[2mm]
&=&(2^{i-1}-2^{n-1})a_{i-1,n}-\sum\limits_{j=i+1}^{n-1}a_{i-1,j}d_{j,n}\nonumber\\[2mm]
&=&(2^{i-1}-2^{n-1})a_{i-1,n}-\sum\limits_{j=i+1}^{n-1}a_{i-1,j}\frac{(2^{j-1}-2^{n-1})d_{11}a_{j-1,n}}{a_{j-1,j}}.
\end{eqnarray}
Due to \eqref{t2} we infer $a_{i-1,j}d_{11}=0$ for any $2\leq i<j<n$.
Keeping this fact and noting $a_{i-1,i}\neq0$ from \eqref{d_in} one finds
$$
d_{i,n}=\frac{(2^{i-1}-2^{n-1})a_{i-1,n}}{a_{i-1,i}},\ \ \ 1<i<n-1.
$$
This equality together with \eqref{d_n-1,n=} implies \eqref{d_i+1,n=matind} for any
$1\leq i\leq n-2$.

So, from \eqref{d_ij=0},\eqref{d_nn},\eqref{tt1},\eqref{tt2} and \eqref{d_i+1,n=matind} we
conclude that $d$ is a derivation of evolution algebra given by \eqref{evolalg}
if and only if

\begin{eqnarray}
\label{ttt1}d_{ij}=d_{ni}=0, & 1\leq i\neq j\leq n-1;\\[2mm]
\label{ttt2}d_{ii}=2^{i-1}d_{11}, & 2\leq i\leq n;\\[2mm]
\label{ttt3}a_{ij}d_{11}=0, & i+2\leq j\leq n-1;\\[2mm]
\label{ttt4}d_{i+1,n}=(2^i-2^{n-1})d_{11}a_{in}, & 1\leq i\leq n-2.
\end{eqnarray}

{\bf Case $I_A\neq\emptyset$}. In this case, we have $a_{i_0j_0}\neq0$ for some pair $(i_0,j_0)$
satisfying $i_0+2\leq j_0<n$.
Then, from \eqref{ttt3} one finds $d_{11}=0$. Plugging this fact into \eqref{ttt2}
and \eqref{ttt4} we obtain
$$
\Der({\bf{E}})=\left\{
\left(
\begin{array}{lllll}
0 & 0 & \ldots & 0 & \beta\\
0 & 0 & \ldots & 0 & 0\\
\vdots & \vdots & \ddots & \vdots & \vdots\\
0 & 0 & \ldots & 0 & 0\\
0 & 0 & \ldots & 0 & 0
\end{array}
\right):\ \beta\in\mathbb K
\right\}.
$$

{\bf Case $I_A=\emptyset$}. In this case \eqref{ttt3} is true for any $d_{11}\in\mathbb K$.
So, from \eqref{ttt1},\eqref{ttt2} and \eqref{ttt4} we conclude that
$$
\Der({\bf{E}})=\left\{
\left(
\begin{array}{lllll}
\alpha & 0 & \ldots & 0 & \beta\\
0 & 2\alpha & \ldots & 0 & (2-2^{n-1})\alpha a_{1n}\\
\vdots & \vdots & \ddots & \vdots & \vdots\\
0 & 0 & \ldots & 2^{n-2}\alpha & (2^{n-2}-2^{n-1})\alpha a_{n-2,n}\\
0 & 0 & \ldots & 0 & 2^{n-1}\alpha
\end{array}
\right):\ \alpha,\beta\in\mathbb K
\right\}.
$$
This completes the proof.
\end{proof}

\begin{rk}
From the proved theorem we infer that $1\leq\dim Der({\bf E})\leq2$. This kind of result
could be proved using Jacobson \cite{Jacob}. But the advantage of Theorem \ref{thm_der}
is that it fully describes structure of the derivations in the natural basis.
\end{rk}

Now it is natural to consider the following question: if the $Der({\bf E})$ is given
is it possible to reconstruct a nilpotent evolution algebra ${\bf E}$. To solve this question we need
an auxiliary fact.

\begin{lemma}\label{lem3}
Let $\bf E$ be an evolution algebra with a natural basis $\{{\bf e}_k\}_{k=1}^n$ and multiplaction
table:
\begin{equation}\label{eval4}
{\bf{e}}_i^2=\left\{
\begin{array}{ll}
a_{i,i+1}{\bf{e}}_{i+1}+a_{in}{\bf{e}}_n, & i<n-1\\
a_{n-1,n}{\bf{e}}_n, & i=n-1\\
{\bf0}, & i=n.
\end{array}
\right.
\end{equation}
If $\dim({\bf E}^2)<n-1$ then the Lie algebra
$\Der({\bf E})$ has dimension more than two.
\end{lemma}

\begin{proof} For the evolution algebra
\eqref{eval4} let us denote
$$
\begin{array}{ll}
I=\{1,2,\dots,n-1\},\\
I_1=\{i\in I: a_{i,i+1}\neq0\},\\
I_2=\{i\in I: a_{i,i+1}=0,\ a_{in}\neq0\},\\
I_3=\{i\in I: a_{i,i+1}=0,\ a_{in}=0\}.
\end{array}
$$
It is clear that $I=I_1\cup I_2\cup I_3$ and $I_i\cap I_j=\emptyset$ for $i\neq j$. We
note that $n-1\not\in I_2$.
Since $\dim({\bf E}^2)<n-1$ we have $I_2\cup I_3\neq\emptyset$. Take an arbitrary derivation $d$ of the algebra ${\bf E}$. Then, due to $d({\bf e}_i^2)=2d_{ii}{\bf e}_i^2$ for all $i\in I$, we obtain
\begin{equation}\label{It1}
\begin{array}{ll}
a_{i,i+1}d_{i+1,i+1}+a_{in}d_{n,i+1}=2a_{i,i+1}d_{ii}, & i<n-1\\
a_{i,i+1}d_{i+1,n}+a_{in}d_{nn}=2a_{in}d_{ii}, & i<n-1\\
a_{i,i+1}d_{i+1,k}+a_{in}d_{nk}=0, & i<n-1,\ k\not\in\{i+1,n\}\\
a_{n-1,n}d_{nn}=2a_{n-1,n}d_{n-1,n-1}\\
a_{n-1,n}d_{nk}=0, & k<n.
\end{array}
\end{equation}
Furthermore, from $d({\bf e}_i{\bf e}_j)=\bf 0$ for all $i\neq j$ one has
\begin{equation}\label{It2}
\begin{array}{ll}
a_{i,i+1}d_{ji}=0, & i<n-1,\ j<n,\\
a_{in}d_{ji}+a_{jn}d_{ij}=0, & i,j\in I\\
d_{ni}=0, & i\not\in I_3
\end{array}
\end{equation}

{\bf Case $I_1=\emptyset$}. In this case from \eqref{It1} and \eqref{It2} one gets
$$
\begin{array}{ll}
d_{n,i+1}=0, & i\in I_2\\
d_{nn}=2d_{ii}, & i\in I_2\\
d_{nk}=0, & i\in I_2,\ k\not\in\{i+1,n\}\\
d_{ji}=-\frac{a_{jn}}{a_{in}}d_{ij}, & i,j\in I_2,\ i\neq j\\
d_{ji}=0, & i\in I_2,\ j\in I_3\\
d_{ni}=0, & i\in I_2.
\end{array}
$$
According to $n-1\not\in I_2$, we infer that
there exist at least three free variables: $d_{1n},d_{nn}$ and $d_{n-1,n-1}$, which
means
that $dim(\Der({\bf E}))\geq3$.

{\bf Case $I_1\neq\emptyset,\ I_3=\emptyset$}.  In this case, we have $I_2\neq\emptyset$.
We note that $n-1\in I_1$, i.e., $a_{n-1,n}\neq0$.
Then, from
\eqref{It1},\eqref{It2} for all $i\in I_1$ one gets
$$
\begin{array}{ll}
d_{i+1,i+1}=2d_{ii}, & i\in I_1\\
d_{i+1,n}=\frac{a_{in}}{a_{i,i+1}}(2d_{ii}-d_{nn}), & i\in I_1\setminus\{n-1\}\\
d_{i+1,k}=0, & i\in I_1,\ \forall k\not\in\{i+1,n\}\\
d_{ji}=0, & i\in I_1\setminus\{n-1\},\ j\in I,\ j\neq i\\
a_{jn}d_{ij}+a_{in}d_{ji}=0, & i,j\in I,\ j\neq i\\
d_{nn}=2d_{ii}, & i\in I_2
\end{array}
$$
From last ones we conclude that again one can find at least three values $d_{1n}, d_{nn}$ and
$d_{i_0+1,n}$ for some $i_0\in I_2$, which are free variables, hence, $dim(\Der({\bf E}))\geq3$.

{\bf Case $I_1\neq\emptyset,\ I_3\neq\emptyset$}. First, we suppose that $n-1\in I_1$, i.e.,
$a_{n-1,n}\neq0$. Then from \eqref{It1},\eqref{It2} we obtain
$$
\begin{array}{ll}
d_{i+1,i+1}=2d_{ii}, & i\in I_1\\
d_{i+1,n}=\frac{a_{in}}{a_{i,i+1}}(2d_{ii}-d_{nn}), & i\in I_1\setminus\{n-1\}\\
d_{i+1,k}=0, & i\in I_1,\ k\not\in\{i+1,n\}\\
d_{ji}=0, & i\in I_1\setminus\{n-1\},\ j\in I,\ i\neq j\\
a_{in}d_{ji}+a_{jn}d_{ij}=0, & i,j\in I,\ i\neq j\\
d_{nn}=2d_{ii}, & i\in I_2
\end{array}
$$

From last ones we conclude that there exist
at least three values $d_{1n}, d_{nn}$ and
$d_{i_0+1,n}$ for some $i_0\in I_3$, which are free variables. Hence, $dim(\Der({\bf E}))\geq3$.

Let us assume that $n-1\in I_3$, i.e., $a_{n-1,n}=0$.
Then from \eqref{It1},\eqref{It2} we obtain
$$
\begin{array}{ll}
d_{i+1,i+1}=2d_{ii}-\frac{a_{in}}{a_{i,i+1}}d_{n,i+1}, & i\in I_1\\
d_{i+1,n}=\frac{a_{in}}{a_{i,i+1}}(2d_{ii}-d_{nn}), & i\in I_1\\
d_{i+1,k}=-\frac{a_{in}}{a_{i,i+1}}d_{nk}, & i\in I_1,\ k\not\in\{i+1,n\}\\
d_{ji}=0, & i\in I_1,\ j\in I,\ i\neq j\\
a_{in}d_{ji}+a_{jn}d_{ij}=0, & i,j\in I,\ i\neq j\\
d_{n,i+1}=0, & i\in I_2\\
d_{nk}=0, & k<n,\ k-1\not\in I_2\\
d_{ni}=0, & i\not\in I_3\\
d_{nn}=2d_{ii}, & i\in I_2
\end{array}
$$
If $I_2=\emptyset$ then we have at least three values $d_{1n}, d_{11}$ and $d_{nn}$ which are independent.
If $I_2\neq\emptyset$ then we have at least three values $d_{1n}, d_{nn}$ and $d_{i_0+1,n}$ for $i_0\in I_2$
which are independent.
\end{proof}

Now we are ready to formulate result related to the posed question.

\begin{thm}
Let $\bf{E}$ be an evolution algebra with a natural basis $\{{\bf e}_i\}_{i=1}^n$.
Assume that
\begin{equation}\label{teskari_Der}
\Der({\bf{E}})=\left\{
\left(
\begin{array}{lllll}
\alpha & 0 & \ldots & 0 & \beta\\
0 & 2\alpha & \ldots & 0 & \alpha d_1\\
\vdots & \vdots & \ddots & \vdots & \vdots\\
0 & 0 & \ldots & 2^{n-2}\alpha & \alpha d_{n-2}\\
0 & 0 & \ldots & 0 & 2^{n-1}\alpha
\end{array}
\right):\ \alpha,\beta\in\mathbb K
\right\}
\end{equation}
where $\{d_i\}_{i=1}^{n-2}$ are fixed numbers. Then ${\bf E}$ is a nilpotent evolution algebra
with maximal index of nilpotency. Moreover, its multiplication table is given by
\begin{equation}\label{eval2}
{\bf e}_i^2=\left\{
\begin{array}{lll}
a_{i,i+1}{\bf e}_{i+1}+\frac{a_{i,i+1}d_{i+1}}{2^i-2^{n-1}}{\bf e}_n, & i<n-2\\
a_{n-1,n}{\bf e}_n, & i=n-1\\
{\bf 0}, & i=n
\end{array}
\right.
\end{equation}
Here $a_{i,i+1}\neq0$ for every $i<n$.
\end{thm}

\begin{proof}
Let $\bf{E}$ be an evolution algebra with structural matrix $A=(a_{ij})_{i,j\geq1}^n$. Suppose that
$\Der({\bf{E}})$ is given by \eqref{teskari_Der}. Let us fix an arbitrary derivation $d\in\Der({\bf{E}})$ such that
$d_{11}\neq0$. Then from $d({\bf{e}}_i^2)=2d_{ii}{\bf e}_i^2$, we have
$$
\sum_{j=1}^{n-1}a_{ij}d_{jj}{\bf e}_j+\sum_{j=1}^{n}a_{ij}d_{jn}{\bf e}_n=2d_{ii}\sum_{j=1}^na_{ij}{\bf e}_j.
$$
From the last equality for any $i\leq n$ one finds
\begin{equation}\label{ad=2da}
\begin{array}{ll}
a_{ij}d_{jj}=2a_{ij}d_{ii}, & i\neq j<n\\[2mm]
\sum\limits_{k=1}^na_{ik}d_{kn}=2a_{in}d_{ii}
\end{array}
\end{equation}
Since $d_{jj}=2^{j-i}d_{ii}$ from the first equality of \eqref{ad=2da} one gets
$a_{ij}d_{ii}=0$. Noting $d_{11}\neq0$ we have
\begin{equation}\label{a_ij=0}
a_{ij}=0,\ \ \ \forall j\not\in\{i+1,n\}
\end{equation}
Rewrite the second equality of \eqref{ad=2da} for $i=n$ as
\begin{equation}\label{t1t1}
\sum_{k=1}^na_{nk}d_{kn}=2a_{nn}d_{nn}.
\end{equation}
Putting \eqref{a_ij=0} into \eqref{t1t1} we find $a_{nn}=0$.

So, we have shown that a multiplication table of ${\bf E}$ is given by
\begin{equation}\label{eval3}
{\bf{e}}_i^2=\left\{
\begin{array}{ll}
a_{i,i+1}{\bf{e}}_{i+1}+a_{in}{\bf{e}}_n, & i<n-1\\
a_{n-1,n}{\bf{e}}_n, & i=n-1\\
{\bf0}, & i=n.
\end{array}
\right.
\end{equation}
It is clear that $dim(\Der({\bf E}))=2$, then due to Lemma \ref{lem3} one has $a_{i,i+1}\neq0$ for any $i<n$.
Applying this fact, Theorem \ref{thm_der} in the second equality of \eqref{ad=2da} one gets
$$
a_{in}=\frac{a_{i,i+1}d_{i+1,n}}{2d_{ii}-d_{nn}}=\frac{a_{i,i+1}d_{i}}{2^i-2^{n-1}},\ \ i<n-1
$$
Finally putting the last one into \eqref{eval3} we obtain \eqref{eval2}.

This completes the proof.
\end{proof}

Due to Lemmas \ref{lem1} and \ref{lem2} from the last theorem we obtain the following results:

\begin{cor}
If the derivation algebra of evolution algebras is given by \eqref{teskari_Der}, then these evolution
algebras are isomorphic.
\end{cor}

\begin{cor} The
Lie algebras
$$
{\bf E}=\left\{
\left(
\begin{array}{lllll}
\alpha & 0 & \vdots & 0 & \beta\\
0 & 2\alpha & \vdots & 0 & 0\\
\vdots & \vdots & \ddots & \vdots & \vdots\\
0 & 0 & \vdots & 2^{n-2}\alpha & 0\\
0 & 0 & \vdots & 0 & 2^{n-1}\alpha
\end{array}\right):\ \alpha,\beta\in\mathbb K
\right\}
$$
and
$$
{\bf E}'=\left\{
\left(
\begin{array}{lllll}
\alpha & 0 & \vdots & 0 & \beta\\
0 & 2\alpha & \vdots & 0 & \alpha d_1\\
\vdots & \vdots & \ddots & \vdots & \vdots\\
0 & 0 & \vdots & 2^{n-2}\alpha & \alpha d_{n-2}\\
0 & 0 & \vdots & 0 & 2^{n-1}\alpha
\end{array}
\right):\ \alpha,\beta\in\mathbb K
\right\}
$$
are isomorphic for any $d_{i}\in\mathbb K, i=\overline{1,n-2}$.
\end{cor}

\begin{rk}
We stress that isomorphisms of Lie algebras does not imply isomorphism of the
corresponding evolution algebras
(see Example \ref{ex_2}).
\end{rk}

\section{Local and $2$-local derivations for evolution algebras}

The results of section 3 will allow us to describe local and $2$-local
derivations of nilpotent evolution algebra. In this section, we want to fully describe local and 
$2$-local derivations of
nilpotent evolution algebras
with maximal index of nilpotency.

Recall that a linear mapping $\Delta$ on $\bf E$ is called {\it local derivation}
if for every ${\bf u}\in\bf E$ there is a derivation $d_{\bf u}$ such that
$\Delta({\bf u})=d_{\bf u}(\bf u)$. A mapping (not necessary linear) $D:{\bf E}\to{\bf E}$
is called {\it $2$-local derivation} of algebra $\bf E$ if for every ${\bf u},{\bf v}\in{\bf E}$
there exists a derivation $d_{{\bf u},{\bf v}}$ of $\bf E$ such that $D({\bf u})=d_{{\bf u},{\bf v}}({\bf u})$
and $D({\bf v})=d_{{\bf u},{\bf v}}({\bf v})$.

Therefore, it is natural to find all local derivations of $\bf E$.

\begin{thm}\label{thm_localder} Let $\bf E$ be an $n$-dimensional nilpotent evolution algebra with maximal index of nilpotency.
Then the following statements hold:
\begin{enumerate}
\item[$(i)$] If $n=2$, then the space of all local derivations has the following form:
\begin{equation}\label{locder1}
\left\{
\left(
\begin{array}{ll}
\alpha & \beta\\
0 & 2\alpha
\end{array}\right),\ \
\left(
\begin{array}{ll}
\alpha & \beta\\
0 & 0
\end{array}\right): \alpha,\beta\in\mathbb K
\right\}
\end{equation}
\\
\item[$(ii)$]
If $n>2$ then every local derivation of $\bf E$ is a derivation.
\end{enumerate}
\end{thm}

\begin{proof}
$(i)$ Let $n=2$. Then due to Lemma \ref{lem1} we may assume that an evolution algebra $\bf E$ is
given by ${\bf e}_1^2={\bf e}_2$ and ${\bf e}_2^2=\bf0$.
Take an arbitrary linear map $\Delta$ on $\bf E$, i.e.,
$$
\Delta({\bf u})=(\Delta_{11}u_1+\Delta_{21}u_2){\bf e}_1+(\Delta_{12}u_1+\Delta_{22}u_2){\bf e}_2,\ \
\forall {\bf u}=u_1{\bf e}_1+u_2{\bf e}_2.
$$
If $\Delta$ is local derivation then for any $\bf u$ there exist $\alpha_{\bf u}$ and $\beta_{\bf u}$ such that
$$
\begin{array}{ll}
\Delta_{11}u_1+\Delta_{21}u_2=\alpha_{\bf u}u_1\\
\Delta_{12}u_1+\Delta_{22}u_2=\beta_{\bf u}u_1+2\alpha_{\bf u}u_2
\end{array}
$$
From the first equation we get $\Delta_{21}=0$. If we take $\bf u$ such that $u_1=0$ then
from the second equation we immediately find $\Delta_{22}\in\{0,2\Delta_{11}\}$. It
is easy to that $\Delta$ is a derivation of $\bf E$ if $\Delta_{22}=2\Delta_{11}$.

Suppose $\Delta_{22}=0$ and $\Delta_{11}\neq0$.
Then for every $\bf u$ we can find derivation $d_{\bf u}$ satisfying $\Delta({\bf u})=d_{\bf u}({\bf u})$
as follows
$$
d_{\bf u}=\left\{
\begin{array}{lr}
\left(
\begin{array}{ll}
\Delta_{11} & \Delta_{12}-\frac{2\Delta_{11}u_2}{u_1}\\
0 & 2\Delta_{11}
\end{array}\right), & \mbox{if }u_1\neq0,\\[5mm]
\left(
\begin{array}{ll}
0 & 0\\
0 & 0
\end{array}\right), & \mbox{if }u_1=0.
\end{array}\right.
$$
This means that a linear mapping defined by
\begin{equation}\label{locder2}
\Delta=\left(
\begin{array}{ll}
\alpha & \beta\\
0 & 2\alpha
\end{array}\right),\ \ \alpha,\beta\in\mathbb K
\end{equation}
is a local derivation of $\bf E$.

Finally, since every derivation of algebra $\bf E$ is local derivation and due to
\eqref{locder2} one gets \eqref{locder1}

$(ii)$
Let $\Delta$ be a non zero local derivation given by matrix $(\Delta_{ij})_{i,j\geq1}^n$.
Assume that $I_A\neq\emptyset$. Then due to $\Delta({\bf e}_i)=d_{{\bf e}_i}({\bf e}_i)$ for any $i\leq n$
we immediately get $\Delta_{1n}=\beta_{{\bf e}_1}$ and $\Delta_{ij}=0$ otherwise. It yields that
$\Delta\in\Der({\bf E})$.

Suppose that $I_A=\emptyset$. Let us establish that $\Delta\in\Der({\bf E})$ for any local
derivation $\Delta$. Due to Lemmas \ref{lem1} and \ref{lem2}
in order to show every local derivation can be derivation
it is enough to check only for evolution algebra ${\bf E}'$ (see Lemma \ref{lem1}).

Since $\Delta({\bf e}_i)=d_{{\bf e}_i}({\bf e}_i)$ for any $i\leq n$
we can easily find
\begin{equation}\label{Delta_ii}
\begin{array}{ll}
\Delta_{ii}=d_{ii}^{({\bf e}_i)}, & i\leq n\\
\Delta_{1n}=d_{1n}^{({\bf e}_1)}\\
\Delta_{ij}=0, & \mbox{otherwise}
\end{array}
\end{equation}
Taking ${\bf u}=\sum_{k=1}^{n-1}{\bf e}_k$ we obtain
\begin{equation}\label{Delta}
\Delta_{ii}=2^{i-1}\Delta_{11},\ \  i<n
\end{equation}
Consider ${\bf v}={\bf e}_2+{\bf e}_n$. Then there exists a derivation
$d_{{\bf v}}$ such that $\Delta({\bf v})=d_{{\bf v}}({\bf v})$. Due to the
assumption $(ii)$ of Theorem \ref{thm_der}
we have
$$
\Delta_{22}{\bf e}_2+\Delta_{nn}{\bf e}_n=2d_{11}^{({\bf v})}{{\bf e}_2}+2^{n-1}d_{11}^{({\bf v})}{\bf e}_n.
$$
This implies that
$$
\begin{array}{ll}
2d_{11}^{({\bf v})}=\Delta_{22}\\
2^{n-1}d_{11}^{({\bf v})}=\Delta_{nn}
\end{array}
$$
Plugging last ones into \eqref{Delta} we obtain $\Delta_{ii}=2^{i-1}\Delta_{11}$. Then using \eqref{Delta_ii}
one finds
$$
\begin{array}{ll}
\Delta_{ii}=d_{ii}^{({\bf e}_1)}, & i\leq n\\
\Delta_{1n}=d_{1n}^{({\bf e}_1)}\\
\Delta_{ij}=0, & \mbox{otherwise}
\end{array}
$$
Hence, due to Theorem \ref{thm_der}
we conclude that $\Delta$ is a derivation.

This completes the proof.
\end{proof}

\begin{thm}\label{thm_2localder}
Every $2$-local derivation of nilpotent evolution algebras
with maximal index of nilpotency is a derivation.
\end{thm}

\begin{proof} Let $D$ be a non zero $2$-local derivation of $\bf E$. Denote $\Gamma=\{{\bf u}\in{\bf E}: u_1\neq0\}$.

{\bf Case} $I_A=\emptyset$. By definition there exist functionals
$\alpha_{{\bf u},{\bf v}}$ and $\beta_{{\bf u},{\bf v}}$
such that
\begin{equation}\label{D}
\begin{array}{ll}
D({\bf u})=\sum\limits_{k=1}^{n-1}2^{k-1}\alpha_{{\bf u},{\bf v}}u_k{\bf e}_k+
\left(\beta_{{\bf u},{\bf v}}u_1+2^{n-1}\alpha_{{\bf u},{\bf v}}u_{n}+\alpha_{{\bf u},{\bf v}}\sum\limits_{k=1}^{n-2}d_ku_{k+1}\right){\bf e}_n\\[3mm]
D({\bf v})=\sum\limits_{k=1}^{n-1}2^{k-1}\alpha_{{\bf u},{\bf v}}v_k{\bf e}_k+
\left(\beta_{{\bf u},{\bf v}}v_1+2^{n-1}\alpha_{{\bf u},{\bf v}}v_{n}+\alpha_{{\bf u},{\bf v}}\sum\limits_{k=1}^{n-2}d_kv_{k+1}\right){\bf e}_n
\end{array}
\end{equation}
where ${\bf u}=\sum_{k=1}^nu_k{\bf e}_k$ and ${\bf v}=\sum_{k=1}^nv_k{\bf e}_k$.

Take an arbitrary
non-zero ${\bf u}\in\bf E$. Then, for any ${\bf v},{\bf v}'\in\bf E$ from the last ones we find
$$
\begin{array}{ll}
\sum\limits_{k=1}^{n-1}2^{k-1}\alpha_{{\bf u},{\bf v}}u_k{\bf e}_k+
\left(\beta_{{\bf u},{\bf v}}u_1+2^{n-1}\alpha_{{\bf u},{\bf v}}u_{n}+\alpha_{{\bf u},{\bf v}}\sum\limits_{k=1}^{n-2}d_ku_{k+1}\right){\bf e}_n\\[3mm]
=\sum\limits_{k=1}^{n-1}2^{k-1}\alpha_{{\bf u},{\bf v}'}u_k{\bf e}_k+
\left(\beta_{{\bf u},{\bf v}'}u_1+2^{n-1}\alpha_{{\bf u},{\bf v}'}u_{n}+\alpha_{{\bf u},{\bf v}'}\sum\limits_{k=1}^{n-2}d_ku_{k+1}\right){\bf e}_n
\end{array}
$$
which is equivalent to
$$
\begin{array}{lr}
\alpha_{{\bf u},{\bf v}}u_k=\alpha_{{\bf u},{\bf v}'}u_k, & k=\overline{1,n-1}\\[3mm]
\beta_{{\bf u},{\bf v}}u_1+2^{n-1}\alpha_{{\bf u},{\bf v}}u_{n}+\alpha_{{\bf u},{\bf v}}\sum\limits_{k=1}^{n-2}d_ku_{k+1}\\[3mm]
=\beta_{{\bf u},{\bf v}'}u_1+2^{n-1}\alpha_{{\bf u},{\bf v}'}u_{n}+\alpha_{{\bf u},{\bf v}'}\sum\limits_{k=1}^{n-2}d_ku_{k+1}
\end{array}
$$
Since, ${\bf u}\neq\bf0$ we get
$\alpha_{{\bf u},{\bf v}}=\alpha_{{\bf u},{\bf v}'}$ for any ${\bf v},{\bf v}'\in\bf E$. This means
that
\begin{equation}\label{albet}
\alpha_{{\bf u},{\bf v}}=:\alpha_{{\bf u}}
\end{equation}
Moreover, if ${\bf u}\in\Gamma$ then one finds
\begin{equation}\label{albet2}
\beta_{{\bf u},{\bf v}}=:\beta_{{\bf u}}
\end{equation}
Taking \eqref{albet},\eqref{albet2} into \eqref{D} we conclude that mapping $D$
can be defined as follows
\begin{equation}\label{Ddef}
D({\bf u})=\left\{
\begin{array}{ll}
\sum\limits_{k=1}^n2^{k-1}\alpha_{{\bf u}}u_k{\bf e}_k+\left(\beta_{{\bf u}}u_1+\alpha_{\bf u}\sum\limits_{k=1}^{n-2}d_ku_{k+1}\right){\bf e}_n, & \mbox{if }{\bf u}\in\Gamma\\[4mm]
\sum\limits_{k=2}^{n-1}2^{k-1}\alpha_{{\bf u}}u_k{\bf e}_k+\left(2^{u-1}u_n+\sum\limits_{k=1}^{n-2}d_ku_{k+1}\right)\alpha_{{\bf u}}{\bf e}_n, & \mbox{if }{\bf u}\not\in\Gamma
\end{array}\right.
\end{equation}
Then for any ${\bf u}',{\bf v}'\in\bf E$ we can find
derivation $d$ given by
$$
d_{ij}=\left\{
\begin{array}{ll}
2^{i-1}\alpha, & \mbox{if }i=j\\
\beta, & \mbox{if }i=1, j=n\\
\alpha d_{i-1}, & \mbox{if }1<i<n, j=n\\
0, & \mbox{otherwise}
\end{array}
\right.
$$
such that
$D({\bf u}')=d({\bf u}')$,
$D({\bf v}')=d({\bf v}')$. Then from \eqref{Ddef}
one gets
\begin{equation}\label{alp=alp}
\alpha_{{\bf u}'}=\alpha_{{\bf v}'}=\alpha\ \ \ \mbox{for any }{\bf u}',{\bf v}'\in\bf E.
\end{equation}
This means that functional $\alpha_{\bf u}$ is a constant.

To complete the proof we show $\beta_{\bf u}=const$ for any $\bf u\in\bf E$.
Let us consider non-zero points ${\bf u},{\bf v}\in\Gamma$. Using the first equality of \eqref{Ddef} and noting
\eqref{alp=alp} by definition of $2$-local derivation we get $\beta_{\bf u}=\beta_{\bf v}$.
This means that $\beta_{\bf u}$ does not depend $\bf u$, i.e., $\beta_{\bf u}=\beta$ for any $\bf u\in\Gamma$.
Putting this fact and \eqref{alp=alp} into \eqref{Ddef} yields that $D$ has the following form

$$
D({\bf u})=\sum_{k=1}^{n-1}2^{k-1}\alpha u_k{\bf e}_k+\left(\beta u_1+2^{n-1}\alpha u_n+\alpha\sum\limits_{k=1}^{n-2}d_ku_{k+1}\right){\bf e}_n.
$$
Due to Theorem \ref{thm_der} $(ii)$ $D$ is a derivation.

{\bf Case} $I_A\neq\emptyset$.
By definition there exist functionals $\alpha_{{\bf u},{\bf v}}$ and $\beta_{{\bf u},{\bf v}}$ such that
\begin{equation}\label{D2}
\begin{array}{ll}
D({\bf u})=\beta_{{\bf u},{\bf v}}u_1{\bf e}_n\\[2mm]
D({\bf v})=\beta_{{\bf u},{\bf v}}v_1{\bf e}_n
\end{array}
\end{equation}

Take arbitrary ${\bf u}\in\Gamma$.
Then from the first equation of \eqref{D2} we obtain
$\beta_{{\bf u},{\bf v}}=\beta_{{\bf u},{\bf v}'}$ for any ${\bf v},{\bf v}'\in{\bf E}$. This means that
$\beta_{{\bf u},{\bf v}}$ does not depend on $\bf v$, i.e.,
$\beta_{{\bf u},{\bf v}}=\beta_{{\bf u}},\ \forall{\bf u}\in\Gamma$. On the other
hand, from the second equation of \eqref{D2} we get $\beta_{{\bf u},{\bf v}}=\beta_{\bf v}$ for any
${\bf v}\in\Gamma$. These facts yield that $\beta_{\bf u}=:\beta$ for any ${\bf u},{\bf v}\in\bf E$. Consequently,
we have
$$
D({\bf u})=\beta u_1{\bf e}_1.
$$
Due to Theorem \ref{thm_der} $(i)$
we obtain $D\in\Der({\bf E})$.

\end{proof}

\section{Automorphisms and local automorphisms}

Recall that by an {\it automorphism} of an evolution algebra
$\bf E$ we mean an isomorphism of $\bf E$ into itself. The set of all
automorphisms is denoted by $Aut({\bf E})$. It is known that
$Aut({\bf E})$ is a group. In this section we are going
to describe $Aut({\bf E})$ of nilpotent evolution algebras
with maximal index of nilpotency.

If $I_A\neq\emptyset$, then by $\eta$ we denote the largest common divisor
of all numbers $2^{j-1}-2^{i}$ where $(i,j)\in I_A$, i.e.,
\begin{equation}\label{eta}
\eta=LCD_{(i,j)\in I_A}(2^{j-1}-2^i)
\end{equation}

\begin{thm}\label{thm_auto}
Let $\bf E$ be an $n$-dimensional nilpotent evolution algebra with maximal index of nilpotency
and $A=(a_{ij})_{i,j=1}^n$ be its structural matrix in a natural basis $\{{\bf e}_i\}_{i=1}^n$.
Then the following statements hold:
\begin{enumerate}
\item[$(i)$] if $I_A\neq\emptyset$ then
$$
Aut({\bf{E}})=\left\{
\left(
\begin{array}{lllll}
\alpha & 0 & \ldots & 0 & \beta\\
0 & \alpha^2 & \ldots & 0 & \varphi_{2n}\\
\vdots & \vdots & \ddots & \vdots & \vdots\\
0 & 0 & \ldots & \alpha^{2^{n-2}} & \varphi_{n-1,n}\\
0 & 0 & \ldots & 0 & \alpha^{2^{n-1}}
\end{array}
\right):\ \alpha,\beta\in\mathbb K,\ \alpha^\eta=1
\right\}
$$
where $\eta$ is defined as \eqref{eta}, and $\varphi_{in}$ is given by the following recurrence formula
$$
\begin{array}{ll}
\varphi_{n-1,n}=a_{n-2,n}(\alpha^{2^{n-2}}-\alpha^{2^{n-1}}),\\[2mm]
\varphi_{n-i,n}=a_{n-i-1,n}(\alpha^{2^{n-i-1}}-\alpha^{2^{n-1}})-\sum\limits_{k=1}^{i-1}a_{n-i-1,n-k}\varphi_{n-k,n}, & 1<i<n-1.
\end{array}
$$
\item[$(ii)$] if $I_A=\emptyset$ then
$$
Aut({\bf{E}})=\left\{
\left(
\begin{array}{lllll}
\alpha & 0 & \ldots & 0 & \beta\\
0 & \alpha^2 & \ldots & 0 & a_{1,n}(\alpha^2-\alpha^{2^{n-1}})\\
\vdots & \vdots & \ddots & \vdots & \vdots\\
0 & 0 & \ldots & \alpha^{2^{n-2}} & a_{n-2,n}(\alpha^{2^{n-2}}-\alpha^{2^{n-1}})\\
0 & 0 & \ldots & 0 & \alpha^{2^{n-1}}
\end{array}
\right):\ \alpha,\beta\in\mathbb K,\ \alpha\neq0
\right\}
$$
\end{enumerate}
\end{thm}

\begin{proof}
Let $\varphi$ be a linear mapping on $\bf{E}$. Now we represent
$\varphi$ on the basis elements as follows:
$$
\varphi({\bf e}_i)=\sum_{j=1}^n\varphi_{ij}{\bf e}_j,\ \ \ 1\leq i\leq n.
$$
We want to describe matrix $(\varphi_{ij})_{i,j=1}^n$ when $\varphi$ is an automorphism of ${\bf E}$.
Suppose that $\varphi$ is an automorphism.
Then we have
$$
\begin{array}{ll}
\varphi({\bf e}_i)\varphi({\bf e}_j)={\bf 0}, & i\neq j\\[2mm]
\varphi({\bf e}_i^2)=[\varphi({\bf e}_i)]^2, & 1\leq i\leq n
\end{array}
$$
which is equivalent to the followings:

\begin{eqnarray}
\label{phi1}\sum_{k=1}^{n-1}\varphi_{ik}\varphi_{jk}{\bf e}_k^2={\bf 0},\ \  i\neq j\\
\label{phi2}\sum_{j=i+1}^na_{ij}\sum_{k=1}^n\varphi_{jk}{\bf e}_k=
\sum_{k=1}^{n-1}\varphi_{ik}^2{\bf e}_k^2,\ \ i\leq n-2\\
\label{phi3}a_{n-1,n}(\sum_{k=1}^n\varphi_{nk}{\bf e}_k)=\sum_{k=1}^{n-1}\varphi_{n-1,k}^2{\bf e}_{k}^2,\\
\label{phi4}\sum_{k=1}^{n-1}\varphi_{nk}^2{\bf e}_k^2=0.
\end{eqnarray}

The linear independence of $\{{\bf e}_1^2,{\bf e}_2^2,\cdots,{\bf e}_{n-1}^2\}$ together
with \eqref{phi1},\eqref{phi4} implies

\begin{eqnarray}
\label{varphi11}\varphi_{ik}\varphi_{jk}=0, & i\neq j,\ k\leq n-1\\
\label{varphi_nk=0}\varphi_{nk}=0, & k\leq n-1
\end{eqnarray}

We notice that
$\varphi_{nn}\neq0$. Now plugging \eqref{varphi_nk=0} into \eqref{phi3} one finds
\begin{equation}\label{varphi_nn=}
\begin{array}{ll}
\varphi_{nn}=\varphi_{n-1,n-1}^2\\
\varphi_{n-1,k}=0, &  k\leq n-2
\end{array}
\end{equation}

Inserting ${\bf e}_l^2=\sum_{j=l+1}^na_{lj}{\bf e}_j,\ l\leq n-1$ into \eqref{phi2} we obtain
\begin{eqnarray}
\label{varphi13}\sum_{j=i+1}^na_{ij}\varphi_{jl}=\sum_{j=1}^{l-1}a_{jl}\varphi_{ij}^2, &  i\leq n-2,\ l\geq2\\
\label{varphi14}\sum_{j=i+1}^na_{ij}\varphi_{j1}=0, & i\leq n-2
\end{eqnarray}

We claim:
\begin{equation}\label{matind2}
\begin{array}{ll}
\varphi_{il}=0, &  l+1\leq i\\
\varphi_{j+1,j+1}=\varphi_{jj}^2, & j\leq n-1
\end{array}
\end{equation}
Let us prove the last relations by induction. Due to \eqref{varphi_nk=0},\eqref{varphi_nn=}
the first step is satisfied.
Take an arbitrary $i_0>1$ and assume that for any $i>i_0$ assertion \eqref{matind2} holds.

We must prove that $\varphi_{i_0l}=0$ for any $l\leq i_0-1$ and $\varphi_{i_0i_0}=\varphi_{i_0-1,i_0-1}^2$.
Rewriting \eqref{varphi13} for $i=i_0>1$ one finds
\begin{equation}\label{varphii0j}
\sum_{j=i_0+1}^na_{i_0j}\varphi_{jl}=\sum_{j=1}^{l-1}a_{jl}\varphi_{i_0j}^2,\ \ l\geq2
\end{equation}
If $j>i_0$ then due to the assumption we have $\varphi_{jl}=0$ for any $l\leq i_0$. So,  for any
$l\leq i_0$ the left side of
\eqref{varphii0j} equals to zero.
Hence,
\begin{equation}\label{matind3}
\sum_{j=1}^{l-1}\varphi_{i_0j}^2a_{jl}=0,\ \ 2\leq l\leq i_0
\end{equation}
If $l=2$ then form \eqref{matind3} we obtain $a_{12}\varphi_{i_01}=0$. Noting $a_{12}\neq0$ one has $\varphi_{i_0,1}=0$.
Suppose that $\varphi_{i_0,l}=0$ for every $l<l_0\leq i_0$. Then this fact together with
\eqref{matind3} for $l=l_0$ implies $a_{l_0-1,l_0}\varphi_{i_0,l_0}=0$.
From $a_{l_0-1,l_0}\neq0$ it follows
$\varphi_{i_0,l_0}=0$. Thus, we have shown that $\varphi_{i_0,l}=0$ for every $l\leq i_0$.
From the arbitraryness of $i_0>1$ we conclude that
\begin{equation}\label{mi1}
\varphi_{il}=0,\ \ l+1<i
\end{equation}

On the other hand, rewriting \eqref{varphi13} for $l=i+1$ and keeping in mind \eqref{mi1} one gets
$$
a_{i,i+1}\varphi_{i+1,i+1}=a_{i,i+1}\varphi_{ii}^2,\ \ \ i\leq n-2.
$$
Due to $a_{i,i+1}\neq0$, the last equality yields $\varphi_{i+1,i+1}=\varphi_{ii}^2$ for every $i\leq n-2$.
This together with \eqref{varphi_nn=} implies
\begin{equation}\label{varphi_ii=}
\varphi_{ii}=\varphi_{11}^{2^{i-1}}\neq0,\ \ \ i\leq n.
\end{equation}
Hence, from \eqref{mi1} and \eqref{varphi_ii=} it follows \eqref{matind2}.

Plugging \eqref{mi1} into \eqref{varphi11} we have
\begin{equation}\label{varphi_ij=0}
\varphi_{ij}=0,\ \ \ i<j<n.
\end{equation}

Let us consider \eqref{varphi13} for $l>i+1$. Then for every $i\leq n-2$ we obtain

\begin{eqnarray}
\label{eq1}a_{il}\varphi_{ll}=a_{il}\varphi_{ii}^2,\ \ i+1<l<n\\[2mm]
\label{eq2}\sum\limits_{j=i+1}^na_{ij}\varphi_{jn}=a_{in}\varphi_{ii}^2,\ \ l=n.
\end{eqnarray}

From \eqref{eq2} with \eqref{varphi_ii=} we get a recurrence formula for $\varphi_{i,n}$ as follows:
\begin{equation}\label{reccur}
\begin{array}{ll}
\varphi_{n-1,n}=a_{n-2,n}(\varphi_{11}^{2^{n-2}}-\varphi_{11}^{2^{n-1}}),\\[2mm]
\varphi_{n-i,n}=a_{n-i-1,n}(\varphi_{11}^{2^{n-i-1}}-\varphi_{11}^{2^{n-1}})-\sum\limits_{k=1}^{i-1}a_{n-i-1,n-k}\varphi_{n-k,n}, & 1<i<n-1.
\end{array}
\end{equation}

Hence, we infer that $\varphi$ is an automorphism of evolution algebra \eqref{evolalg} if and only
if the followings hold:
\begin{equation}\label{varphi_auto}
\begin{array}{ll}
\varphi_{ij}=0, & i\neq j,\ j<n\\[2mm]
\varphi_{ii}=\varphi_{11}^{2^{i-1}}, & i\leq n\\[2mm]
a_{il}\varphi_{ll}=a_{il}\varphi_{ii}^2, & i+1<l<n\\[2mm]
\varphi_{n-1,n}=a_{n-2,n}(\varphi_{11}^{2^{n-2}}-\varphi_{11}^{2^{n-1}}),\\[2mm]
\varphi_{n-i,n}=a_{n-i-1,n}(\varphi_{11}^{2^{n-i-1}}-\varphi_{11}^{2^{n-1}})-\sum\limits_{k=1}^{i-1}a_{n-i-1,n-k}\varphi_{n-k,n}, & 1<i<n-1.
\end{array}
\end{equation}
Now let us consider two cases w.r.t.$I_A$.

{\bf Case $I_A=\emptyset$}. For the sake of convenience, we denote $\varphi_{11}=\alpha\neq0$.
Then from \eqref{varphi_auto} one gets
$$
\begin{array}{llll}
\varphi_{ij}=\varphi_{ji}=0, & i\neq j,\ j<n\\[2mm]
\varphi_{ii}=\alpha^{2^{i-1}}, & i\leq n\\[2mm]
\varphi_{1n}=\beta,\\[2mm]
\varphi_{in}=a_{i-1,n}(\alpha^{2^{i-1}}-\alpha^{2^{n-1}}) , & 1<i<n
\end{array}
$$
where $\beta\in\mathbb K$, which yields the assertion.

{\bf Case $I_A\neq\emptyset$}.
Then for the automorphism $\varphi$ we have
$$
\begin{array}{llll}
\varphi_{ij}=\varphi_{ji}=0, & i\neq j,\ j<n\\[2mm]
\varphi_{ii}=\alpha^{2^{i-1}}, & 1\leq i\leq n\\[2mm]
\alpha^{2^{l-1}-2^{i}}=1, & (i,l)\in I_A\\[2mm]
\varphi_{1n}=\beta,\\[2mm]
\varphi_{n-1,n}=a_{n-2,n}(\alpha^{2^{n-2}}-\alpha^{2^{n-1}}),\\[2mm]
\varphi_{n-i,n}=a_{n-i-1,n}(\varphi_{11}^{2^{n-i-1}}-\varphi_{11}^{2^{n-1}})-\sum\limits_{k=1}^{i-1}a_{n-i-1,n-k}\varphi_{n-k,n}, & 1<i<n-1.
\end{array}
$$
where $\alpha,\beta\in\mathbb K$ and $\alpha^\eta=1$, which implies the assertion.

The proof is complete.
\end{proof}

\begin{cor}\label{eta=2}
Let $I_A\neq\emptyset$ and $\eta=2$. Then $\varphi$
 is an automorphism of the evolution algebra \eqref{evolalg} iff it has the following form:
 \begin{equation}\label{eta=2_varphi}
 \varphi=\left(
 \begin{array}{llllll}
 \alpha & 0 & 0 & \vdots & 0 & \beta\\
 0 & 1 & 0 & \vdots & 0 & 0\\
  \vdots & \vdots & \vdots &\ddots & 0 & 0\\
 0 & 0 & 0 & \vdots & 1 & 0\\
 0 & 0 & 0 & \vdots & 0 & 1\\
  \end{array}\right)
 \end{equation}
where $\beta\in\mathbb K$ and $\alpha^2=1$.
\end{cor}
\begin{ex}\label{ex_2}
Now we show the existence of two nilpotent evolution algebras with $\dim({\bf{E}}^2)=n-1$ such that
they are not isomorphic to each other, and they have the same group of automorphisms and the Lie algebra of
derivations.

Let $n\geq5$ and consider the evolution algebras ${\bf{E}}_1$ and ${\bf{E}}_2$ given by
$$
{\bf{E}}_1^2:\left\{
\begin{array}{ll}
{\bf e}_1^2={\bf e}_2+{\bf e}_3+{\bf e}_4,\\
{\bf e}_i^2={\bf e}_{i+1}, & 1<i<n,\\
{\bf e}_{n}^2={\bf 0}
\end{array}\right.
$$
$$
{\bf{E}}_2^2:\left\{
\begin{array}{ll}
{\bf f}_1^2={\bf f}_2+{\bf f}_3,\\
{\bf f}_i^2={\bf f}_{i+1}, & 1<i<n,\\
{\bf f}_{n}^2={\bf 0}
\end{array}\right.
$$
By $A_1$ and $A_2$ we denote the matrices of structural constants of these algebras, respectively.
Then, it is easy to see that $I_{A_1}=\{(1,3),(1,4)\}$ and $I_{A_2}=\{(1,3)\}$.
Due to Corollary \ref{eta=2} we have $Aut({\bf E}_1)=Aut({\bf E}_2)$. Moreover,
according to Theorem \ref{thm_der} $(i)$ one can find $\Der({\bf E}_1)=\Der({\bf E}_2)$.

Let us establish that they are not isomorphic. Assume that ${\bf{E}}_1\cong {\bf{E}}_2$ and
$\xi=(\xi_{ij})_{i,j\geq1}^n$ be an isomorphism. Then it is easy to check that
$$
{\bf f}_i=\xi_{ii}{\bf e}_i+\xi_{in}{\bf e}_n,\ \ \ i\geq2
$$
Due to ${\bf f}_1{\bf f}_i=0$ for any $i\geq2$ one has ${\bf f}_{1}=\xi_{11}{\bf e}_1+\xi_{1n}{\bf e}_n$.
From ${\bf f}_1^2={\bf f}_2+{\bf f}_3$
we find
$$
\begin{array}{ll}
\xi_{11}^2=\xi_{22}\\
\xi_{11}^2=\xi_{33}\\
\xi_{11}^2=0\\
\xi_{2n}=-\xi_{3n}
\end{array}
$$
This contradicts to $\det(\xi)\neq0$. So, we infer that ${\bf{E}}_1\not\cong{\bf{E}}_2$.
\end{ex}

\subsection{Local automorphisms of Evolution algebras}

In previous section we have been able to find the set of all
automorphisms of evolution algebra \eqref{evolalg}. Now we will show that every
local automorphism is automorphism if evolution algebra is defined by \eqref{evolalg} with $n>2$.
Recall that a linear mapping $\psi$ from $\bf E$ to $\bf E$ is called {\it local automorphism}
if for every $\bf u\in\bf E$ there exists an automorphism
$\varphi_{\bf u}\in Aut({\bf E})$ such that $\psi({\bf u})=\varphi_{\bf u}({\bf u})$.

\begin{thm}
Let $\bf E$ be an $n$-dimensional nilpotent evolution algebra with maximal index of nilpotency.
Then the following statements hold:
\begin{enumerate}
\item[$(i)$] If $n=2$, then the set of all local automorphisms has the following form:
\begin{equation}\label{locder11}
\left\{
\left(
\begin{array}{ll}
\alpha & \beta\\
0 & \gamma^2
\end{array}\right): \alpha,\beta,\gamma\in\mathbb K,\ \alpha\gamma\neq0
\right\}
\end{equation}
\\
\item[$(ii)$]
If $n>2$ then every local automorphism of $\bf E$ is an automorphism.
\end{enumerate}
\end{thm}
\begin{proof}
$(i)$ Let $n=2$. Then due to Lemma \ref{lem1} we may assume that an evolution algebra $\bf E$ is
given by ${\bf e}_1^2={\bf e}_2$ and ${\bf e}_2^2=\bf0$.
Take an arbitrary linear map $\psi$ on $\bf E$, i.e.,
$$
\psi({\bf u})=(\psi_{11}u_1+\psi_{21}u_2){\bf e}_1+(\psi_{12}u_1+\psi_{22}u_2){\bf e}_2,\ \
\forall {\bf u}=u_1{\bf e}_1+u_2{\bf e}_2.
$$
If $\psi$ is local automorphism then for any $\bf u$ there exist $\alpha_{\bf u}$ and $\beta_{\bf u}$ such that
$$
\begin{array}{ll}
\psi_{11}u_1+\psi_{21}u_2=\alpha_{\bf u}u_1\\
\psi_{12}u_1+\psi_{22}u_2=\beta_{\bf u}u_1+\alpha^2_{\bf u}u_2
\end{array}
$$
From the first equation we get $\psi_{21}=0$. If we take $\bf u$ such that $u_1=0$ then
from the second equation we immediately find $\psi_{22}=\alpha^2_{\bf u}$. It yields that
if $\psi$ is local automorphism it has the following form
\begin{equation}\label{n=2aut}
\left(
\begin{array}{ll}
\alpha & \beta\\
0 & \gamma^2
\end{array}\right)
\end{equation}
where $\alpha\gamma\neq0$.

Let us show that \eqref{n=2aut} is indeed local automorphism of \eqref{evolalg}. In fact, for any
$\bf u\in\bf E$ we may take an automorphism $\varphi_{\bf u}$ of \eqref{evolalg} as follows:
$$
\varphi_{\bf u}=\left\{
\begin{array}{ll}
\left(
\begin{array}{ll}
\alpha & \beta+\frac{(\gamma^2-\alpha^2)u_2}{u_1}\\
0 & \alpha^2
\end{array}
\right), & \mbox{if }u_1\neq0\\[3mm]
\left(
\begin{array}{ll}
\gamma & 0\\
0 & \gamma^2
\end{array}
\right), & \mbox{if }u_1=0
\end{array}
\right.
$$
From this, one can check that $\psi({\bf u})=\varphi_{\bf u}({\bf u})$.

$(ii)$ Let $n>2$.
Let $\psi$ be a local automorphism for \eqref{evolalg}. By definition of local
automorphism, for every ${\bf u}\in{\bf E}$ we have $\psi({\bf u})=\varphi_{\bf u}({\bf u})$,
where $\varphi_{\bf u}$ is automorphism.
Then Theorem \ref{thm_auto} implies $\psi_{ij}=0$ for every $i\neq j$, $j<n$. On the other hand,
taking ${\bf u}={\bf e}_i$, $i\leq n$, we conclude that the local automorphism
$\psi$ has the following form:
$$
\psi=\left(
\begin{array}{llllll}
\alpha_{{\bf e}_1} & 0 & 0 & \vdots & 0 & \beta_{{\bf e}_1}\\
0 & \alpha_{{\bf e}_2}^2 & 0 & \vdots & 0 & \varphi_{2n}^{({\bf e}_2)}\\
\vdots & \vdots & \vdots & \ddots & \vdots & \vdots\\
0 & 0 & 0 & \vdots & \alpha_{{\bf e}_{n-1}}^{2^{n-2}} & \varphi_{n-1,n}^{({\bf e}_{n-1})}\\
0 & 0 & 0 & \vdots & 0 & \alpha_{{\bf e}_{n}}^{2^{n-1}}
\end{array}
\right)
$$
Now we take arbitrary ${\bf v}=\sum_{i=1}^{n}v_i{\bf e}_i$. Then from $\psi({\bf v})=\varphi_{\bf v}({\bf v})$ one gets
\begin{eqnarray}
\label{psi11}&&\alpha_{{\bf e}_i}^{2^{i-1}}v_i=\alpha_{\bf v}^{2^{i-1}}v_i,\ \  i<n\\
\label{psi22}&&\beta_{{\bf e}_1}v_1+\alpha_{{\bf e}_n}^{2^{n-1}}v_n+\sum_{k=2}^{n-1}\varphi_{kn}^{({\bf e}_k)}v_k=
\beta_{{\bf v}}v_1+\alpha_{{\bf v}}^{2^{n-1}}v_n+\sum_{k=2}^{n-1}\varphi_{kn}^{({\bf v})}v_k
\end{eqnarray}
From \eqref{psi11} we find
\begin{equation}\label{psiii}
\alpha_{{\bf e}_i}^{2^{i-1}}=\alpha_{{\bf e}_1}^{2^{i-1}},\ \ \ \ i<n
\end{equation}
Consequently,
$\varphi_{kn}^{({\bf e}_k)}=\varphi_{kn}^{({\bf e}_1)}$ for any $k<n$. Keeping in mind this fact, from \eqref{psi22}
one gets
\begin{equation}\label{ufff}
\beta_{{\bf e}_1}v_1+\alpha_{{\bf e}_n}^{2^{n-1}}v_n=
\beta_{{\bf v}}v_1+\alpha_{{\bf v}}^{2^{n-1}}v_n
\end{equation}
Finally, taking ${\bf v}'={\bf e}_2+{\bf e}_n$, form \eqref{psiii},\eqref{ufff}
we obtain
$$
\begin{array}{ll}
\alpha_{{\bf e}_1}^2=\alpha_{{\bf v}'}^2\\
\alpha_{{\bf e}_n}^{2^{n-1}}=\alpha_{{\bf v}'}^{2^{n-1}}
\end{array}
$$
which yields $\alpha_{{\bf e}_n}^{2^{n-1}}=\alpha_{{\bf e}_1}^{2^{n-1}}$. Putting the last one into \eqref{ufff} we get
$\beta_{{\bf e}_1}=\beta_{\bf v}$ for any $\bf v\in\bf E$.

So, we conclude that local automorphism $\psi=(\varphi_{ij})$ has the following form:
$$
\varphi_{ij}=\left\{
\begin{array}{ll}
\alpha_{{\bf e}_1}^{2^{i-1}}, & i=j\\[2mm]
\beta_{{\bf e}_1}, & i=1,\ j=n\\[2mm]
\varphi_{in}^{({\bf e}_1)}, & i>1,\ j=n\\[2mm]
0, & \mbox{otherwise}
\end{array}
\right.
$$
Hence, Theorem \ref{thm_auto} implies that the local automorphism $\psi$ is an automorphism.
This completes the proof.
\end{proof}

\section*{Acknowledgments}
The present work is supported by the UAEU "Start-Up" Grant, No.
31S259.

\end{document}